\documentclass[oneside, 11pt]{amsart}
\usepackage[utf8]{inputenc}

\usepackage{amsfonts}
\usepackage{amsmath}
\usepackage{amssymb}
\usepackage{amsthm}
\usepackage[a4paper]{geometry}
\usepackage{mathrsfs} 
\usepackage[dvipsnames]{xcolor}
\usepackage[pagebackref=false,colorlinks=true, linkcolor=blue, citecolor=blue]{hyperref}

\usepackage{appendix}
\usepackage{enumerate}
\usepackage{geometry}
\usepackage{tikz-cd} 
\usepackage{orcidlink}

\theoremstyle{plain}
\newtheorem{theorem}{Theorem}
\newtheorem{proposition}[theorem]{Proposition}
\newtheorem{lemma}[theorem]{Lemma}

\newtheorem{remark}[theorem]{Remark}

\theoremstyle{definition}

\numberwithin{theorem}{section}
\numberwithin{equation}{section} 

\newcommand{\norm}[1]{\left \lVert  #1 \right \rVert}
\newcommand{\abs}[1]{\left\lvert #1 \right\rvert}

\newcommand{\vertiii}[1]{{\left\vert\kern-0.25ex\left\vert\kern-0.25ex\left\vert #1 
    \right\vert\kern-0.25ex\right\vert\kern-0.25ex\right\vert}}

\newcommand{\Z}{\mathbb{Z}}

\newcommand{\R}{\mathbb{R}}
\newcommand{\N}{\mathbb{N}}

\newcommand{\calF}{\mathcal{F}}

\newcommand{\calM}{\mathcal{M}}

\title[Restriction and mixing properties]{Restriction and mixing properties of interacting particle systems with unbounded range}
\author{Benedikt Jahnel \ \orcidlink{0000-0002-4212-0065}}
\address{Technische Universit\"at Braunschweig \& Weierstrass Institute, Berlin, Germany. }
\email{benedikt.jahnel@tu-braunschweig.de}
\author{Jonas K\"oppl \ \orcidlink{0000-0001-9188-1883}}
\address{Weierstrass Institute, Berlin, Germany. }
\email{koeppl@wias-berlin.de}
\date{\today}
\begin{document}

\begin{abstract} 
    We consider interacting particle systems with unbounded interaction range on general countably infinite graphs $S$ and prove explicit non-asymptotic error bounds for approximations of the infinite-volume dynamics by systems of finitely many interacting particles. Moreover, we also provide non-asymptotic quantitative bounds on the spatial decay of correlations at times $t>0$ and then apply these results to show that interacting particle systems on $\Z$ whose interaction strengths decays exponentially fast cannot spontaneously break the time-translation symmetry, neither in the strong, nor in the weak sense. 
\end{abstract}
\maketitle
\section{Introduction}
We consider interacting particle systems, which are Markov processes on the state space $\Omega = \{0,\dots,q-1\}^S$, where the set of sites is some countably infinite set $S$ equipped with a metric $d$, specified in terms of generators of the form 
\begin{align*}
    \mathscr{L}f(\eta) = \sum_{\Delta \Subset S}\sum_{\xi_\Delta\in \Omega_\Delta}c_\Delta(\eta,\xi_\Delta)\left[f(\xi_\Delta \eta_{\Delta^c})-f(\eta)\right], \quad \eta \in \Omega, 
\end{align*}
for local functions $f\colon\Omega \to \R$. We write $\Delta \Subset S$ to signify that $\Delta$ is a \textit{finite} subset of $S$ and $\Omega_\Delta:=\{0,\dots, q-1\}^\Delta$. The transition rates $c_\Delta(\eta,\xi_\Delta)$ can be interpreted as the infinitesimal rate at which the particles inside of the finite volume $\Delta$ switch from the local state $\eta_\Delta$ to $\xi_\Delta$, given that the rest of the system is currently in the state $\eta_{\Delta^c}$. We will denote the associated semigroup by $(S(t))_{t\geq0}$ and the set of probability measures on $\Omega$ by $\mathcal{M}_1(\Omega)$. 

\medskip 
A prototypical example of such a system is the Glauber dynamics for the Ising model, i.e., single-site spin-flip dynamics on the configuration space $\Omega = \{\pm 1\}^S$ with rates given by 
\begin{align}\label{example-ising-glauber}
    c_x(\eta,-\eta_x) =  \left[1+\exp\left(\mathcal{H}(-\eta_x \eta_{x^c})-\mathcal{H}(\eta)\right)\right]^{-1},
\end{align}
where 
\begin{align*}
    \mathcal{H}(\omega) = -\sum_{x,y \in S}J_{x,y}\omega_x \omega_y, \quad \omega \in \Omega. 
\end{align*}
In many physical applications, the coupling constants $J_{x,y}$ are not strictly finite range but instead exhibit a spatial decay, e.g., power-law or exponential. 

While finite-range systems allow for intuitive graphical representations via Harris' construction, see, e.g., \cite[Chapter~4]{swart_course_2026}, these unbounded-range dependencies complicate the picture. Since the Hamiltonian $\mathcal{H}$, and thus also the rates, depends on particles at arbitrary distances, classical results regarding the finite-speed-of-propagation, see, e.g., \cite[Lemma~3.2]{martinelli_lectures_1999}, do not apply. Therefore, we have to proceed differently and rely on the analytic tools that the general existence theory for interacting particle systems, as laid out in \cite[Chapter~I]{liggett_interacting_2005}, provides. \medskip 

In this article, we address three questions related to the behaviour of interacting particle systems with unbounded interaction range: 
\begin{enumerate}[\bfseries (Q1)]
    \item \textbf{Finite-volume approximation:} If we only observe the process in a fixed finite volume $\Lambda \Subset S$ until some time $t>0$, how well can the true infinite-volume dynamics  $(S(t))_{t\geq 0}$ be approximated by a finite-volume system which only performs updates in a finite region $\Lambda^h := \{x \in S\colon d(x,\Lambda) \leq h\}$. In particular, how large do we need to choose $h=h(t)>0$ to obtain a sufficiently good approximation?
    \item \textbf{Spatial decay of correlations: }If we observe the infinite-volume dynamics in two distant volumes $\Lambda_1, \Lambda_2 \Subset S$, how strongly are these two parts of the system correlated at some finite time $t$? In particular, how fast do dependencies spread in the system? 
    \item \textbf{Long-time behaviour:} Last but not least, what can the approximation via finite systems tell us about the infinite-volume dynamics? In particular, is it possible to transport certain facts about the long-time behaviour of finite systems to sufficiently well-behaved infinite volume systems? 
\end{enumerate}

While the classical Trotter--Kurtz theorem, see \cite[Theorem~2.12 and Corollary~3.10]{liggett_interacting_2005}, establishes the \textit{qualitative} convergence of finite-volume restrictions to the infinite-volume dynamics, such results are typically asymptotic in nature. In contrast, we provide \textit{non-asymptotic quantitative} estimates. These bounds explicitly characterise the approximation error in terms of the time $t$ and the speed of decay of the interactions. 

Structurally, the error bounds we derive can be viewed as classical stochastic analogues of the Lieb--Robinson bounds found in quantum spin systems, see e.g., \cite{nachtergaele_lieb-robinson_2010, lieb_finite_1972}. Just as Lieb--Robinson bounds define a \textit{light cone} within which information can propagate in a quantum lattice spin system, our estimates quantify the spatial spread of dependencies in a classical interacting particle system. By bounding the influence on and of distant sites, we effectively establish a rigorous control on the speed of information propagation, even in the presence of possibly long-range interactions. 

As a primary application of our approximation result, we investigate the long-time behaviour of interacting particle systems on $\Z$. We show that for interacting particle systems with exponentially decaying interaction strengths the attractor of the measure-valued dynamics is equal to the set of measures which are stationary with respect to the dynamics. In particular, non-trivial time-periodic behaviour is impossible in such systems. 
This result extends previous works of Mountford \cite{mountford_coupling_1995} and Ramirez--Varadhan \cite{ramirez_relative_1996} to systems with unbounded interaction range. This extension is particularly noteworthy because it is known that the conclusion of our theorem does not hold in dimensions $d\geq 3$, see \cite{jahnel_class_2014, jahnel_time-periodic_2025}, highlighting a fundamental phase transition in the dynamics dictated by the underlying spatial geometry.

\subsection*{Organisation of the manuscript}
In Section \ref{section:setting-main-results}, we introduce the precise framework in which we are working and setup the required notation before we state our main results. We then provide brief outline of the proof strategy for one of our main results in the direction of $\mathbf{(Q3)}$, namely the strong attractor property for interacting particle systems on $\Z$, in Section \ref{section:proof-strategy}. After this, we start with the main work and provide the proofs of the results related to questions $\mathbf{(Q1)}$ and $\mathbf{(Q2)}$. The proof of the strong attractor property is then proved via two steps in Section \ref{section:proof-refined-restriction} and Section \ref{section:proof-speed-up}. We end the paper with a short outlook and some open problems in Section \ref{section:outlook}.

\section{Setting and main results}\label{section:setting-main-results}
Let $S$ be a countably infinite set of sites, equipped with a metric $d$, and for $q \in \N$ define the product space $\Omega := \Omega_o^{S}$ := $\{0,\dots,q-1\}^{S}$, which we will equip with the usual product topology and the corresponding Borel sigma-algebra $\mathcal{F}$. For $\Lambda \subset S$ let $\calF_\Lambda$ be the sub-sigma-algebra of $\calF$ that is generated by the open sets in $\Omega_\Lambda := \{1,\dots, q\}^{\Lambda}$. We will use the shorthand notation $\Lambda \Subset S$ to signify that $\Lambda$ is a \textit{finite} subset of $S$. If $\Lambda \Subset S$ and $f\colon\Omega \to \R$ is $\mathcal{F}_\Lambda$-measurable, then we will also say that $f$ is $\Lambda$-local. In the following we will often denote for a given configuration $\omega \in \Omega$ by $\omega_\Lambda$ its projection to the volume $\Lambda \subset S$ and write $\omega_\Lambda \omega_\Delta$ for the configuration on $\Lambda \cup \Delta$ composed of $\omega_\Lambda$ and $\omega_\Delta$ for disjoint $\Lambda, \Delta \subset S$. For the special case $\Lambda = \{x\}$ we will also write $x^c = S \setminus \{x\}$ and $\omega_x\omega_{x^c}$. 
The set of probability measures on $\Omega$ will be denoted by $\calM_1(\Omega)$ and the space of continuous functions $f\colon \Omega\to\R$ by $C(\Omega)$. For a configuration $\eta \in \Omega$ we will denote by $\eta^{x,i}$ the configuration that is equal to $\eta$ everywhere except at the site $x$ where it is equal to $i$. Moreover, for $\Lambda \subset S$ we will denote the corresponding cylinder sets by 
$
    [\eta_\Lambda] = \{\omega \colon \omega_\Lambda \equiv \eta_\Lambda \}. 
$
Whenever we are taking the probability of such a cylinder event with respect to some measure $\nu \in \calM_1(\Omega)$, we will omit the square brackets and simply write $\nu(\eta_\Lambda)$.  

\subsection{Interacting particle system}
 We will consider time-continuous Markov dynamics on $\Omega$, namely interacting particle systems characterised by time-homogeneous generators $\mathscr{L}$ with domain $\text{dom}(\mathscr{L})$ and its associated Markovian semigroup $(S(t))_{t \geq 0}$. 
For interacting particle systems we adopt the notation and exposition of the standard reference \cite[Chapter~1]{liggett_interacting_2005}. 
In our setting the generator $\mathscr{L}$ is given via a collection of translation-invariant transition rates $c_\Delta(\eta, \xi_\Delta)$, in finite volumes $\Delta \Subset S$, which are continuous in the starting configuration $\eta \in \Omega$. 
These rates can be interpreted as the infinitesimal rate at which the particles inside $\Delta$ switch from the configuration $\eta_\Delta$ to $\xi_\Delta$, given that the rest of the system is currently in state $\eta_{\Delta^c}$. 
The full dynamics of the interacting particle system is then given as the superposition of these local dynamics, i.e., 
\begin{align*}
    \mathscr{L}f(\eta) = \sum_{\Delta \Subset S}\sum_{\xi_\Delta\in \Omega_\Delta}c_\Delta(\eta, \xi_\Delta)[f(\xi_\Delta \eta_{\Delta^c}) - f(\eta)].
\end{align*}
In \cite[Chapter~1]{liggett_interacting_2005} it is shown that the following two conditions are sufficient to guarantee well-definedness of the dynamics. 
\begin{enumerate}[\bfseries (L1)]
    \item The total rate at which the particle at a particular site changes its spin is uniformly bounded, i.e.,
    \begin{align*}
        \mathbf{C}_1 := \sup_{x\in S}\sum_{\Delta \ni x} \sum_{\xi_{\Delta}\in \Omega_\Delta}\norm{c_{\Delta}(\cdot, \xi_{\Delta})}_{\infty} < \infty
    \end{align*}
    \item and the total influence of a single coordinate on all other coordinates is uniformly bounded, i.e.,
    \begin{align*}
        \mathbf{M}_\gamma := \sup_{x \in S}\sum_{y \neq x}\gamma(x,y) := \sup_{x \in S}\sum_{y \neq x}\sum_{\Delta \ni x}\sum_{\xi_{\Delta}}\delta_y\left(c_{\Delta}(\cdot, \xi_{\Delta})\right) < \infty, 
    \end{align*}
    where 
    \begin{align*}
        \delta_x(f) := \sup_{\eta, \xi\colon \eta_{x^c} = \xi_{x^c}}\abs{f(\eta)-f(\xi)}
    \end{align*}
    is the oscillation of a function $f\colon\Omega \to \R$ at the site $x$. 
\end{enumerate}
Under these conditions one can then show that the operator $\mathscr{L}$, defined as above, is the generator of a well-defined Markov process and that a core of $\mathscr{L}$ is given by 
\begin{align*}
    D(\Omega) := \Big\{ f \in C(\Omega)\colon \sum_{x \in S} \delta_x(f) < \infty\Big\}.
\end{align*}
For $x \in S$ and $\Delta \Subset S$ we introduce the short-hand notation 
\begin{align*}
    \delta_x c_{\Delta} = \sum_{\xi_{\Delta}\in \Omega_\Delta}\delta_x(c_{\Delta}(\cdot, \xi_{\Delta})).
\end{align*}
This measures the influence of the $x$-coordinate on the rate of change in the finite volume $\Delta$. Therefore the quantity
\begin{align*}
    \gamma(y,x) := 
    \begin{cases}
        \sum_{\Delta \ni y}\delta_x c_{\Delta}, \quad &\text{if }x \neq y\\\
        0, &\text{otherwise,}
    \end{cases}
\end{align*}
can be interpreted as the total influence of the $x$-coordinate on the rate of change of the spin at site $y \in S$. Now consider the Banach space $\ell^1(S)$ of all functions $\beta\colon S \to \R$ such that 
\begin{align*}
    \norm{\beta}_{\ell^1} := \sum_{x \in S}\abs{\beta(x)} < \infty. 
\end{align*}
Note that for all $f \in D(\Omega)$ we have $\delta_{\cdot}(f) \in \ell^1(S)$ and we define $\vertiii{f} := \norm{\delta_{\cdot}(f)}_{\ell^1}$. For $\beta \in \ell^1(S)$ we will denote the support of $\beta$ by $\Lambda_\beta$, i.e., 
\begin{align*}
    \Lambda_\beta := \{x \in S\colon \beta(x) \neq 0\}. 
\end{align*}
By a slight abuse of notation, we will also write $\Lambda_f$ for the support of $\delta_\cdot f$ for $f \in C(\Omega)$. This is the set of sites on which the observable $f$ depends. In particular, $f$ is always $\mathcal{F}_{\Lambda_f}$-measurable. 

\subsection{Main results}
For our results we will sometimes also make the following additional assumptions on the maximal size of the update regions. 

\begin{enumerate}[\bfseries (R1)]
    \item The maximal update size is bounded, i.e., there is a constant $L>0$ such that if $\text{diam}(\Delta) \geq L$, then $c_\Delta(\cdot, \cdot) \equiv 0$. 
\end{enumerate}
Additionally, we want to quantify the rate at which $\gamma(x,y)$ tends to zero as a function of the distance $d(x,y)$ between the sites. For this, let $\varrho\colon[0,\infty) \to (0,\infty)$ be a non-increasing function and consider the following assumptions. 
\begin{enumerate}[\bfseries (R2)]
    \item There exists a constant $\mathbf{C}_\varrho \in (0,\infty)$ such that the following inequality holds
    \begin{align*}
        \varrho(d(x,z))\varrho(d(z,y)) \leq \mathbf{C}_\varrho \varrho(d(x,y)), \quad x,y,z \in S.
    \end{align*}
\end{enumerate}
\begin{enumerate}[\bfseries (R3)]
    \item The transition rates, or rather their oscillations, satisfy the decay condition
    \begin{align*}
        \mathbf{C}_\gamma := \sup_{x\in S}\sum_{y\in S}\frac{\gamma(x,y)}{\varrho(d(x,y))} < \infty. 
    \end{align*}
\end{enumerate}
\begin{enumerate}[\bfseries (R4)]
    \item We have 
    \begin{align*}
        \norm{\varrho} := \sup_{x \in S}\sum_{y\in S}\varrho(d(x,y)) < \infty.  
    \end{align*}
\end{enumerate}
In the case $S = \Z^d$ equipped with the standard $\ell^1$-distance, the functions $\varrho(r) = (1+r)^{-\alpha}$ satisfy the condition $\mathbf{(R2)}$ above for any $\alpha \geq 1$ and $\mathbf{(R4)}$ for any $\alpha > d$. 
Moreover, for any $\mu \geq 0$ and any $\varrho$ satisfying the above conditions, the function $\varrho_\mu(r) := {\rm e}^{-\mu r}\varrho(r)$ also satisfies the above conditions with $\norm{\varrho_\mu} \leq \norm{\varrho}$ and $\mathbf{C}_{\varrho_\mu}\leq \mathbf{C}_\varrho$.  
  
\subsection{Restriction to finite volumes}
For a fixed finite subset $\Lambda \Subset S$ and a length-scale $h>0$ define the $h$-blow-up of $\Lambda$ by $\Lambda^h := \{u \in S\colon \text{dist}(u,\Lambda) \leq h\}$,
where $$\text{dist}(\Delta,\Lambda):=\inf\{d(u,v)\colon u\in\Delta, v\in \Lambda\},$$ 
and the generator of the dynamics restricted to $\Lambda^h$ by 
\begin{align*}
    \mathscr{L}^h f(\eta) = \sum_{\Delta \subset \Lambda_h}\sum_{\xi_\Delta\in \Omega_\Delta}c_\Delta(\eta,\xi_\Delta)\left[f(\xi_\Delta \eta_{\Delta^c})-f(\eta)\right], \quad \eta \in \Omega, f \in D(\Omega). 
\end{align*}
So only the particles inside of the finite volume $\Lambda_h$ participate in the dynamics, whereas all other particles remain fixed. If we just observe the dynamics in the smaller volume $\Lambda \subset \Lambda_h \Subset S$, we want to estimate the error we make by considering $\mathscr{L}^h$ instead of $\mathscr{L}$. This also tells us how large we have to choose $h$, depending on the time window $[0,t]$ and the volume $\Lambda$, to get a decent approximation of the infinite-volume dynamics. 
To measure the error, we choose the total variation metric with respect to the sub-sigma-algebra $\mathcal{F}_\Lambda$, i.e., 
\begin{align*}
    d_{\text{TV}, \Lambda}(\nu, \rho) := \sup_{f \text{ $\Lambda$-local},\ \norm{f}_\infty \leq 1}\abs{\nu(f)-\rho(f)}, \quad \nu, \rho \in \mathcal{M}_1(\Omega).  
\end{align*}
The following theorem provides a non-asymptotic estimate on the approximation error made by considering the restricted dynamics instead of the infinite-volume process. Let us further note that all of the constants appearing in the error bound are explicit. 

\begin{theorem}[Restriction to finite volumes]\label{theorem:restriction-to-finite-volumes}
Assume that the conditions $\mathbf{(L1)}$ and $\mathbf{(R1)-(R4)}$ are satisfied. Let $\Lambda \Subset S$ and $h>0$. Then, for all $\Lambda$-local functions $f\colon\Omega \to \R$ we have 
\begin{align}
    \norm{S(t)f - S^h(t)f} \leq \norm{f}_\infty \frac{\mathbf{C}_1 \mathbf{C}_{S,L}}{\mathbf{C}_\varrho^2 \mathbf{C}_\gamma}\exp(\mathbf{C}_\gamma \mathbf{C}_\varrho t) \sum_{x \notin \Lambda^{h-L}}\varrho(\emph{dist}(x,\Lambda)),
\end{align}
where the constant $\mathbf{C}_{S,L}$ is a geometric quantity defined by 
\begin{align*}
    \mathbf{C}_{S,L} := \sup_{x \in S}\abs{\{\Delta \Subset S\colon x \in \Delta \text{ and } \emph{diam}(\Delta) < L\}}. 
\end{align*}
In particular, for any initial distribution $\mu \in \calM_1(\Omega)$ the total variation error in $\Lambda$ is bounded by 
\begin{align*}
    d_{\emph{TV},\Lambda}(\mu S(t), \mu S^h(t)) \leq \frac{\mathbf{C}_1 \mathbf{C}_{S,L}}{\mathbf{C}_\varrho^2 \mathbf{C}_\gamma}\exp(\mathbf{C}_\gamma \mathbf{C}_\varrho t) \sum_{x \notin \Lambda^{h-L}}\varrho(\emph{dist}(x,\Lambda)).
\end{align*}
\end{theorem}

This shows that it suffices to estimate the tail sums for $\varrho(d(\cdot, \cdot))$ to get the distance at which one can truncate the process. However, note that if $S= \Z^d$ and $\varrho(r) = (1+r)^{-\alpha}$, then we only get a decaying right-hand side for $\alpha >d$, which means that the dependencies in the rates, i.e., $\gamma(x,y)$,  need to decay faster than $\abs{x-y}^{-(\alpha+d)}$.
Therefore, our result covers systems with power-law interactions, but only up until a certain threshold, depending on the geometry of the underlying graph $S$. \medskip 

A similar error bound can be derived for situations where $h$ is not a fixed constant, but also depends on time. This extension is done in Proposition~\ref{proposition:refined-restriction-estimate} and is crucial for the proof of one of our other main results, namely Theorem~\ref{theorem:main-result}. 

\subsection{Approximation of stationary measures}
For interacting particle systems in infinite volume it is in general notoriously difficult to say anything non-trivial about the stationary measures. For certain classes, e.g., attractive spin systems, one can approximate the stationary measures for the infinite-volume dynamics by the stationary measures for finite-volume dynamics with specific boundary conditions, see, e.g., \cite[Theorem~III.2.7]{liggett_interacting_2005}. The following result gives sufficient conditions for not necessarily attractive interacting particle systems to enjoy a similar approximation property. 

\begin{theorem}[Approximation of stationary measures]\label{theorem:approximation-of-stationary-measures}
    Assume that the conditions $\mathbf{(L1)}$ and $\mathbf{(R1)-(R4)}$ are satisfied. Assume further that there exists a deterministic configuration $\eta$ and a non-increasing function $F\colon[0,\infty) \to \R$ with $F(t) \downarrow 0$ as $t\uparrow\infty$ such that for any $h>0$ there exists $\mu^h \in \calM_1(\Omega)$ such that
    \begin{align*}
        \abs{S^h(t)f(\eta) - \mu^h(f)} \leq C(f) F(t),\quad\forall f\colon\Omega \to \R\text{ local},
    \end{align*}
    where the constant $C(f)$ is allowed to depend on $f$ but not on $h$. 
    Then, the sequence $(\mu^h)_{h \geq 0}$ converges to a limit $\mu^*$ in $\calM_1(\Omega)$ and $\mu^*$ is stationary for the infinite-volume dynamics $(S(t))_{t \geq 0}$. 
\end{theorem}

Note that we only need the uniformity for\textit{ one specific} initial (and in some sense boundary) condition $\eta$ and not for all $\eta \in \Omega$. Therefore, this theorem may also be applied in low temperature situations where convergence to equilibrium in finite volumes is typically not uniform in the system size. 

\subsection{Spatial decay of correlations}
In equilibrium statistical mechanics, the decay of correlations between spins at distant sites is one of the key characteristics. In the situation we are interested in, one can ask very similar questions but additionally has to consider how dependencies may spread in time due to the interactions in the transition rates. We first give a general estimate on the spatial decay of correlations for fixed times $t\geq 0$ and then show how this can be used to obtain information about some stationary measures. 

The following estimate extends and generalises \cite[Proposition~I.4.18]{liggett_interacting_2005} to systems with unbounded range of interaction and possibly non-translation-invariant transition rates. 

\begin{theorem}[Spatial decay of correlations]\label{thm:spatial-decay-correlations-lieb-robinson}
    Assume that $\mathbf{(L1)}$ and $\mathbf{(R1)-(R3)}$ hold. Then, for all $f,g \in D(\Omega)$ and $t\geq 0$ we have 
    \begin{align*}
        \norm{S(t)[fg]-[S(t)f][S(t)g]}_\infty \leq 
        \frac{\mathbf{C}_1}{\varrho(L)\mathbf{C}_\varrho^2 \mathbf{C}_\gamma} \vertiii{f}\vertiii{g}\exp(2 \mathbf{C}_\varrho\mathbf{C}_\gamma t) \varrho(\emph{dist}(\Lambda_f, \Lambda_g)).
    \end{align*}
\end{theorem}

Depending on the decay of $\varrho(\cdot)$, this gives us an upper bound on how far correlations between distant sites have spread due to the interactions up until time $t$. On $S=\Z^d$ for any rate $\alpha >0$  and $\varrho(r) = \exp(-\alpha r)$ this tells us that the speed at which information is being propagated through the system is linear. If $\varrho$ decays like a power law, this theorem only provides an exponential bound on this speed, which is not expected to be optimal. 

\subsection{Quantitative decay of correlations for limiting stationary measures}
The non-asymptotic bound in Theorem~\ref{thm:spatial-decay-correlations-lieb-robinson}  holds rather generally and already tells us something about how fast dependencies spread in the system. However, the right-hand side still depends on $t$ and it does in particular not give us information about the decay of correlations for stationary measures $\mu$ that can be obtained as limits for some fixed initial configuration. 

Under the additional assumption of rapid convergence to equilibrium for some fixed initial condition, we can remove this inhomogeneity and also obtain that, in this case, the limiting measure also satisfies a quantitative decay of correlations estimate. 

\begin{theorem}[Spatial mixing for limiting stationary measures]\label{theorem:decay-of-correlations-limiting-measure}
    Assume that $\mathbf{(L1)}$ and $\mathbf{(R1)-(R3)}$ hold and that for some $\eta \in \Omega$, $\mu \in \calM_1(\Omega)$, and constants $\hat{K}, \hat{\alpha} >0$,
    \begin{align}\label{assumption:rapid-convergence}
        \abs{S(t)f(\eta) - \mu(f)} \leq \hat{K} {\rm e}^{-\hat{\alpha} t}\vertiii{f},\qquad\forall f \in D(\Omega),\ t\geq 0.
    \end{align}
    Then, there exist $K, \alpha >0$ such that for that particular $\eta$ and all $t\geq 0$ we have
    \begin{align*}
        \abs{S(t)(fg)(\eta) - [S(t)f(\eta)][S(t)g(\eta)]} \leq K \vertiii{f}\vertiii{g} \varrho(\emph{dist}(\Lambda_f,\Lambda_g))^{\alpha}
    \end{align*}
    for all $f,g \in D(\Omega)$ with $\emph{dist}(\Lambda_f, \Lambda_g) > L \mathbf{C}_\varrho$. In particular, we get the following bound on the correlations of $\mu$:
    \begin{align*}
        \abs{\mu(fg) - \mu(f)\mu(g)} \leq K \vertiii{f}\vertiii{g} \varrho(\emph{dist}(\Lambda_f,\Lambda_g))^{\alpha} 
    \end{align*}
    for all $f,g \in D(\Omega)$ with $\emph{dist}(\Lambda_f, \Lambda_g) > L \mathbf{C}_\varrho$.
\end{theorem}

Note that the theorem above does not assume uniqueness of the stationary distribution but just the rapid convergence for one particular initial configuration. In particular, this assumption can also be justified in the phase-coexistence regime. This builds a bridge from temporal mixing to spatial mixing without requiring finite-range or exponentially decaying interactions and also works for $\varrho(\cdot)$ that behave like a power law. Since we do not require assumption $\mathbf{(R4)}$ for Theorem~\ref{thm:spatial-decay-correlations-lieb-robinson} and Theorem~\ref{theorem:decay-of-correlations-limiting-measure} they even apply to interacting particle systems on $\Z^d$ with long-range interactions as long as $\gamma(x,y) \lesssim \abs{x-y}^{-\alpha}$ for $\alpha >d$.

\subsection{Absence of time-translation symmetry breaking in one dimension}
Historically, in the literature on interacting particle systems the most attention has been paid to the set of \textit{time-stationary measures} for the dynamics which is given by 
\begin{align*}
    \mathscr{S} = \{\nu \in \calM_1(\Omega)\colon \nu S(t) = \nu\ \forall t\geq 0\}.
\end{align*}
However, if one is interested in the long-term behaviour of an interacting particle system, a more natural and richer object to study is the so-called \textit{attractor} of the measure-valued dynamics, which is defined as 
\begin{align*}
    \mathscr{A} = \big\{\nu \in \calM_1(\Omega)\colon \exists \nu_0 \in \calM_1(\Omega) \text{ and } t_n \uparrow \infty \text{ such that } \lim_{n \to \infty}\nu_{t_n} = \nu \big\}, 
\end{align*}
where the convergence is in the weak sense. 
In other words, $\mathscr{A}$ is the set of all accumulation points of the measure-valued dynamics induced by $\mathscr{L}$. In the language of dynamical systems this is the $\omega$-limit set and it encodes (most of) the dynamically relevant information about the long-time behaviour of the system. In particular, it is the natural object to consider for studying the phenomenon of \textit{spontaneous symmetry breaking} in the context of interacting particle systems. For an interacting particle system we say that a symmetry of the transition rates is spontaneously broken if there is an element of the attractor $\mathscr{A}$ that does not satisfy the symmetry. For example, for the Ising model Glauber dynamics on $S = \Z^d$, two obvious symmetries are the invariance under global spin-flip, i.e., $c_x(\eta, -\eta_x) = c_x(-\eta, \eta_x)$ and under translations, i.e., $c_x(\eta,-\eta_x) = c_0(\tau_x \eta, -(\tau_x \eta)_0)$. It is well known, that both of these symmetries can be spontaneously broken in infinite volume, at least in higher dimensions. Indeed, as shown by Peierls for the spin-flip symmetry in dimensions $d\geq2$ and by Dobrushin for the spatial translation symmetry in $d\geq 3$. 
\medskip 

Another obvious and hence often overlooked symmetry is that the transition rates do not depend on time, i.e., the generator is autonomous, and is thus invariant under time-shifts. Of course, trivially any stationary measure is also invariant under time shifts and one needs to be a bit more careful when defining the notion of time-translation symmetry breaking. This can be done in the language of the attractor. For this, we first introduce another subset of the attractor, namely the measures which lie on a \textit{stationary orbit}, i.e., 
\begin{align*}
    \mathscr{O} := \{\nu \in \mathcal{M}_1(\Omega)\colon  \exists t >0\text{ such that } \nu S(t) = \nu \}.
\end{align*}
It is clear by the definition that $\mathscr{S}\subset \mathscr{O}\subset \mathscr{A}$. We will say that \textit{(strong) time-translation symmetry breaking} occurs if $\mathscr{S} \subsetneq \mathscr{O}$, in other words, there exists a non-trivial time-periodic orbit $(\mu_s)_{s\in[0,\tau]}$ in the space $\mathcal{M}_1(\Omega)$ of probability measures on $\Omega$. This is similar to the crystallisation phase transition, where the continuous-translation symmetry by any vector in $\R^d$ is spontaneously reduced to a discrete translation symmetry. Additionally, we say that \textit{weak time-translation symmetry breaking} occurs if $\mathscr{S} \subsetneq \mathscr{A}$. It is obvious that the strong notion of time-translation symmetry breaking implies the weak one. The converse is not clear and we expect that generally the weak notion does \textit{not} imply the strong one. The conditions under which interacting particle systems can exhibit stable time-periodic behaviour, thereby spontaneously breaking the time-translation symmetry, have been extensively discussed in the physics literature \cite{grinstein_temporally_1993,bennett_stability_1990,chate_collective_1992}. As it turns out, the possibility or non-possibility of spontaneous time-translation symmetry breaking seems to depend heavily on the dimension of the underlying graph. Non-rigorous arguments \cite{grinstein_temporally_1993} and extensive numerical studies \cite{avni_nonreciprocal_2025,avni_dynamical_2025,guislain_collective_2024}
suggest that interacting particle systems with short-range interactions can exhibit strong time-translation symmetry breaking in dimensions $d\geq3$ but cannot produce stable time-periodic behaviour in $d=1,2$. \medskip 

By combining the above results on information propagation and approximation properties with the proof strategy of \cite{ramirez_relative_1996}, we obtain the following generalisation of Mountford's theorem for interacting particle systems on $S=\Z$ with possibly unbounded range. 

\begin{theorem}[Absence of time-translation symmetry breaking]\label{theorem:main-result}
    Let $S = \Z$ be equipped with the $\ell^1$-metric and assume that $\mathscr{L}$ is the generator of an interacting particle system that satisfies assumptions $\mathbf{(L1)}$ and $\mathbf{(R1)-(R3)}$ with $\varrho(r) = \exp(-r\alpha)$ for some $\alpha >0$. Denote the associated Markov semigroup by $(S(t))_{t\geq 0}$. Then, we have 
    \begin{align*}
        \mathscr{A} = \mathscr{O} = \mathscr{S}. 
    \end{align*}
    In particular, such systems cannot exhibit weak or strong time-translation symmetry breaking and $\abs{\mathscr{S}}=1$ implies that the interacting particle system is ergodic. 
\end{theorem}

In other words, every weak limit point of the dynamics is a stationary measure. In $d\geq3$ there are known counterexamples that satisfy the regularity assumptions of our theorem but exhibit strong time-translation symmetry breaking, whereas in $d=2$ the situation is unclear but believed to be similar to the one-dimensional case. If one drops the short-range assumption, there are also counterexamples in $d=1,2$ that exhibit strong time-translation symmetry breaking, see \cite{jahnel_time-periodic_2025}. 
Let us note that we do not require any shift-invariance or reversibility. 

\section{Proof strategy for the absence of time-translation symmetry breaking}\label{section:proof-strategy} 
Let us briefly comment on the strategy for proving Theorem \ref{theorem:main-result} that was first used in this context in \cite{ramirez_relative_1996} and how it relates to the error estimates for restricting interacting particle systems to finite volumes. 
For this, consider a continuous-time Markov chain with generator $L$ on a finite space $\mathcal{X}$. Denote the associated semigroup by $(P(t))_{t\geq0}$. To show that $\mathscr{A}=\mathscr{S}$ it suffices to show that for any initial distribution $\mu$ and any $\tau >0$ we have 
\begin{align*}
    \limsup_{t \to \infty}d_{\text{TV}}(\mu P(t), \mu P(t+\tau)) = 0. 
\end{align*}
Indeed, by the triangle inequality, if this holds, then any possible limit point of the measure-valued dynamics induced by $(P(t))_{t \geq 0}$ on $\calM_1(\Omega)$ has to be a stationary measure for the Markov chain, and hence $\mathscr{A}=\mathscr{S}$.
Now note that one can use Pinsker's inequality to bound the total variation distance between $\mu P(t)$ and $\mu P(t+\tau)$ in terms of the relative entropy, i.e.,
\begin{align*}
    d_{\text{TV}}(\mu P(t), \mu P(t+\tau)) \leq \sqrt{\frac{1}{2}H(\mu P(t+\tau) \lvert \mu P(t))}.
\end{align*}
Additionally, instead of shifting time by $\tau$, we can also interpret $\mu P(t+\tau)$ as the distribution of a suitably sped-up process with generator $L'$ at time $t$, i.e, $\mu P(t+\tau) = \mu P'(t)$ for $L' = (1+\tau/t)L$ and $P'(t) = e^{tL'}$. 
Now, by a Girsanov-type formula for continuous-time Markov chains on finite state spaces, see, e.g., \cite[Proposition~2.6.\ in Appendix~1]{kipnis_scaling_1999}, one checks that the entropic cost of this speed-up can be bounded by 
\begin{align}\label{key-ineq:proof-sketch}
    H(\mu P(t+\tau) \lvert \mu P(t))=H(\mu P'(t) \lvert \mu P(t))
    \leq H(\mathbb{P}'_{[0,t]}\lvert \mathbb{P}_{[0,t]}) 
\leq 
    \mathbf{c} t \left(\frac{\tau}{t}\right)^2 = \frac{\mathbf{c}\tau^2}{t},
\end{align}
where $\mathbb{P}$ (respectively $\mathbb{P}')$ denotes the law of the whole process with initial distribution $\mu$ and generator $L$ (respectively $L'$) and 
$$\mathbf{c} := \max_{x \in  \mathcal{X}}\sum_{y \neq x}L(x,y).$$ 
Since the right-hand side of~\eqref{key-ineq:proof-sketch} goes to $0$ as $t$ tends to infinity, we are done. \medskip

However, via the use of the Girsanov formula, this proof heavily relies on the fact that the sped-up process with generator $L'$ is absolutely continuous with respect to the original process with generator $L$. This is in general only the case if $L$ has uniformly bounded total jump rate, i.e., if $\mathbf{c} <\infty$. For interacting particle systems this is essentially never the case and we therefore have to proceed a bit differently by first restricting the dynamics to finite volumes $\Lambda^h \Subset \Z$ and considering the approximation error made by this restriction. If we are interested in controlling the total variation error up until time $t>0$, then Theorem~\ref{theorem:restriction-to-finite-volumes} suggests that, even in the finite-range case, the best we can do is scaling $h(t) \sim t$. In this case, the corresponding total jump rate $\mathbf{c}(h(t))$ grows like $t^d$ and even in the case $d=1$ we only get an $O(1)$ upper bound from \eqref{key-ineq:proof-sketch}.
But there is still another screw one can turn to make the argument work. Instead of considering a constant speed-up $\lambda \equiv (1+\tau/t)$ one can also perform a time-dependent speed-up $\lambda(\cdot)$ and optimise over all the admissible options. This is just enough to make the argument work in $d=1$ under suitable conditions on the decay of the interaction strength. \medskip 

To rigorously carry out this argument in detail, we first derive a time-dependent generalisation and refinement of the restriction estimate in Section~\ref{section:proof-refined-restriction}. We then estimate the entropic cost of the optimal time-dependent speed-up and put all the ingredients together in Section~\ref{section:proof-speed-up}. Before we get into this business we first provide the proofs of the general restriction estimate in Theorem~\ref{theorem:restriction-to-finite-volumes} and the decay of correlation properties in Theorem~\ref{thm:spatial-decay-correlations-lieb-robinson} and Theorem~\ref{theorem:decay-of-correlations-limiting-measure}.

\section{Restriction property and decay of correlations}

The following lemma is a consequence of the general existence theory developed by Liggett and for example contained in \cite[Theorem~I.3.9]{liggett_interacting_2005}. 

\begin{lemma}\label{lemma:regularity-estimate-via-gamma}
    Assume that $\mathbf{(L1)}$ and $\mathbf{(L2)}$ hold. Then, for $f \in D(\Omega)$ and $t\geq 0$ we have the following regularity estimate for all $x \in S$
    \begin{align*}
        \delta_{x}\left(S(t)f\right) \leq \exp(t\Gamma)[\delta_{\cdot}(f)](x),
    \end{align*}
    where $\Gamma \colon \ell^1(S) \to \ell^1(S)$, defined by 
    \begin{align*}
        \Gamma \beta (x) := \sum_{y \in S}\gamma(y,x)\beta(y), \quad x \in S, 
    \end{align*}
    is a positive and bounded linear operator on $\ell^1(S)$ with operator norm $\norm{\Gamma}_{\emph{op}}= \mathbf{M}_\gamma$.
\end{lemma}

For estimating the speed at which distant parts of the system get correlated by the dynamics, the following estimate from \cite[Proposition~I.4.4]{liggett_interacting_2005} is quite useful. 
\begin{lemma}\label{lemma:ligget-estimate-2}
    Assume that $\mathbf{(L1)}$ and $\mathbf{(L2)}$ hold. Then, for any $f,g\in D(\Omega)$ and $t\geq 0$ we have 
    \begin{align*}
        \norm{S(t)[fg]-[S(t)f][S(t)g]}_\infty 
        \leq 
        \sum_{x,y \in S}\Big[\sum_{\Delta \ni x,y}\sum_{\xi_\Delta}\norm{c_\Delta(\cdot, \xi_\Delta)}_\infty\Big]\int_0^t \big({\rm e}^{s\Gamma}\delta_\cdot f\big)(x)\big({\rm e}^{s\Gamma}\delta_\cdot g\big)(y)  {\rm d}s. 
    \end{align*}
\end{lemma}

In the situations we are interested in, we have $c_\Delta \equiv 0$ for $\text{diam}(\Delta) \geq L$ for some constant $L>0$, so one typically has the slightly less sharp but simpler estimate 
\begin{align}\label{ineq:simplified-liggett-correlation-estimate}
    \norm{S(t)[fg] - [S(t)f][S(t)g]}_\infty 
        \leq \mathbf{C}_1
        \sum_{x,y\in S\colon d(x,y)\leq L}\int_0^t \big({\rm e}^{s\Gamma}\delta_\cdot f\big)(x)\big({\rm e}^{s\Gamma}\delta_\cdot g\big)(y) {\rm d}s.  
\end{align}
These estimates already tell us that, to obtain upper bounds on how fast information spreads in interacting particle systems, it is generally a good idea to study the action of the operator $\Gamma$ and the associated semigroup $(\exp(t\Gamma))_{t\geq 0}$ on $\ell^1(S)$. They will indeed play a key role in the rest of this section. 

\subsection{Propagation of bounds}
Recall that we assume that the coefficients of $\Gamma$ satisfy the bounds
\begin{align*}
    &\mathbf{C}_\gamma = \sup_{x \in S}\sum_{y\in S}\frac{\gamma(x,y)}{\varrho(d(x,y))}<\infty 
    \quad\text{ and }\quad\mathbf{C}_\varrho = \sup_{x,y,z \in S}\frac{\varrho(d(x,z))\varrho(d(z,y))}{\varrho(d(x,y))} < \infty. 
\end{align*}
Let us see what this tells us about the action of the semigroup $(\exp(t\Gamma))_{t\geq 0}$ on $\ell^1(S)$.
By boundedness of $\Gamma$ we can expand $\exp(t\Gamma)$ for any $t\geq 0$ into 
\begin{align*}
    e^{t\Gamma}\beta(x) = \sum_{n \ge  0}\frac{t^n}{n!}[\Gamma^n \beta](x) 
    \sum_{y \in \Z^d}\beta(y)\sum_{n\ge 0} \frac{t^n}{n!}\gamma^{(n)}(y,x) =: \sum_{y \in \Z^d}\beta(y) \gamma_t(y,x). 
\end{align*}
Our first step is to show how the bounds on $\gamma = \gamma_0$ propagate to later times $t>0$. 
\begin{lemma}\label{lemma:propagation-of-bounds}
   Under the assumptions $\mathbf{(R2)-(R3)}$ we have for any $t \geq 0$
   \begin{align*}
       \sup_{x \in S}\sum_{y \in S}\frac{\gamma_t(x,y)}{\varrho(d(x,y))} \leq \mathbf{C}_\varrho^{-1}\exp(\mathbf{C}_\gamma \mathbf{C}_\varrho t). 
   \end{align*}
\end{lemma}

\begin{proof}
    We can first use an induction argument to show that, for any $n \in \N$, one can express the coefficients $\gamma^{(n)}(u,v)$ of the iterated operator $\Gamma^n$ by 
\begin{align*}
    \gamma^{(n)}(u,v) = \sum_{u_1 \in S}\cdots \sum_{u_{n-1} \in S}\gamma(u,u_1)\cdots \gamma(u_{n-1},v). 
\end{align*}
For fixed $x \in S$ and $n \in \N$ the assumption $\mathbf{(R2)}$ for $\varrho(d(\cdot,\cdot))$ implies that for any $x \in S$
\begin{align*}
    \sum_{y \in S}\frac{\gamma^{(n)}(x,y)}{\varrho(d(x,y))} 
    \leq 
    \mathbf{C}_\varrho^{n-1} \sum_{y \in S}\sum_{x_1 \in S}\cdots\sum_{x_{n-1}\in S}\frac{\gamma(x,x_1)}{\varrho(d(x,x_1))}\cdots \frac{\gamma(x_{n-1},y)}{\varrho(d(x_{n-1},y))}.
\end{align*}
So by applying assumption $\mathbf{(R3)}$ $n$~times we have 
\begin{align*}
    \sup_{x \in S}\sum_{y\in S}\frac{\gamma^{(n)}(y,x)}{\varrho(d(y,x))} \leq \mathbf{C}_\gamma^n \mathbf{C}_\varrho^{n-1}. 
\end{align*}
Thus for $t\geq 0$ we obtain
\begin{align*}
    \sup_{x\in S}\sum_{y \in S}\frac{\gamma_t(x,y)}{\varrho(d(x,y))}
    \leq 
    \mathbf{C}_\varrho^{-1}\sum_{n=0}^\infty \frac{t^n}{n!}\mathbf{C}_\gamma^n \mathbf{C}_\varrho^n \leq 
    \mathbf{C}_\varrho^{-1}\exp(\mathbf{C}_\gamma \mathbf{C}_\varrho t), 
\end{align*}
as desired. 
\end{proof}

The bound in Lemma~\ref{lemma:propagation-of-bounds} directly implies the following quantitative bound on the spatial decay of ${\rm e}^{t\Gamma}\beta$ for compactly supported $\beta \in \ell^1(S)$. 
\begin{lemma}\label{lemma:quantitative-spreading-bound-gamma}
    Assume that $\mathbf{(R2)}-\mathbf{(R3)}$ hold. If $\beta \in \ell^1(S)$ has compact support $\Lambda_\beta \Subset S$, then
    \begin{align*}
        {\rm e}^{t\Gamma}\beta(x) \leq \norm{\beta}_\infty \mathbf{C}_\varrho^{-1}\exp(\mathbf{C}_\gamma \mathbf{C}_\varrho t) \varrho(\emph{dist}(x,\Lambda_\beta)). 
    \end{align*}
    This in particular applies to $\beta = (\delta_x f)_{x \in S}$ for local observables $f\colon\Omega \to \R$ that only depend on some finite volume $\Lambda_f \Subset S$. 
\end{lemma}

\subsection{Restriction to finite volumes}
We proceed with the error bound for approximating the infinite-volume dynamics by restrictions to finite volumes. 

\begin{proof}[Proof of Theorem \ref{theorem:restriction-to-finite-volumes}]
    By Duhamel's formula we have 
    \begin{align*}
        S^h(t)f(\eta) - S(t) f(\eta) = \int_0^t \left(S^{h}(s)(\mathscr{L}^h - \mathscr{L})S(t-s)\right)f(\eta) {\rm d}s. 
    \end{align*}
    Now for any $g \in D(\Omega)$ we can use a telescoping trick to obtain the uniform estimate 
    \begin{align*}
       \big|(\mathscr{L}^h - \mathscr{L})g(\eta)\big| 
        \leq 
        \sum_{\Delta \not\subset \Lambda^h}\sum_{\xi_\Delta}\abs{c_\Delta(\eta, \xi_\Delta)\left[g(\xi_\Delta\eta_{\Delta^c})-g(\eta)\right]}
        \leq 
        \mathbf{C}_1 \mathbf{C}_{S,L} \sum_{x \notin \Lambda^{h-L}}\delta_x g, 
    \end{align*}
    where $\mathbf{C}_{S,L}$ uniformly bounds the number of update regions $\Delta$ in which a particular site $x$ is included. 
    For $\Lambda$-local observables $f\colon \Omega \to \R$ we can combine Lemma~\ref{lemma:regularity-estimate-via-gamma} and the quantitative estimate from Lemma~\ref{lemma:quantitative-spreading-bound-gamma} to get  
    \begin{align*}
    \delta_x\big(S(t-s)f\big)
    \leq \big({\rm e}^{(t-s)\Gamma}\delta_\cdot f\big)(x)
    \leq 
    \norm{f}_\infty \mathbf{C}_\varrho^{-1}\exp(\mathbf{C}_\gamma\mathbf{C}_\varrho (t-s))\varrho(\text{dist}(x,\Lambda)). 
\end{align*}
After using that $\norm{S^h(s)h}_\infty \leq \norm{h}_\infty$ for any $h\in C(\Omega)$, we can combine this with an application of a telescoping estimate from above to the function $g= S(t-s)f$ to see that  
\begin{align*}
    \norm{S^h(t)f-S(t)f}_\infty 
    &\leq 
    \int_0^t \norm{\left(S^h(s)(\mathscr{L}^h-\mathscr{L})S(t-s)\right)f}_\infty{\rm d}s
    \\\
    &\leq 
    \int_0^t\norm{(\mathscr{L}^h-\mathscr{L})S(t-s)f}_\infty {\rm d}s
    \\\
    &\leq
    \norm{f}_\infty \mathbf{C}_1 \mathbf{C}_{S,L}\mathbf{C}_\varrho^{-1}\int_0^t \exp(\mathbf{C}_\gamma \mathbf{C}_\varrho (t-s)){\rm d}s
    \sum_{x \notin \Lambda^{h-L}}\varrho(\text{dist}(x,\Lambda))
    \\\
    &= 
    \norm{f}_\infty \frac{\mathbf{C}_1 \mathbf{C}_{S,L}}{\mathbf{C}_\varrho^2 \mathbf{C}_\gamma}\exp(\mathbf{C}_\gamma \mathbf{C}_\varrho t) \sum_{x \notin \Lambda^{h-L}}\varrho(\text{dist}(x,\Lambda)). 
\end{align*}
This finishes the proof. 
\end{proof}

\subsection{Approximation of stationary measures}
We now apply the results of Theorem~\ref{theorem:restriction-to-finite-volumes} to show that one can approximate the stationary measures of the infinite-volume dynamics via the stationary measures of the restricted dynamics. 

\begin{proof}[Proof of Theorem \ref{theorem:approximation-of-stationary-measures}]
    We first show that the sequence $(\mu^h)_{h \geq 0}$ converges to a limit. 
    For this, note that the compactness of $\calM_1(\Omega)$ implies the existence of limit points and we only have to show uniqueness. For this, it suffices to show that for any fixed $\Lambda$-local observable $f\colon\Omega \to \R$ the sequence $(\mu^h(f))_{h \geq 0}$ is a Cauchy sequence.
    To this end, note that for any $t \geq 0$ and $h,k>0$ we can use Theorem~\ref{theorem:restriction-to-finite-volumes} to get 
    \begin{align*}
        \abs{\mu^h(f)-\mu^k(f)} 
        \leq \
        &\abs{\mu^h(f)-S^h(t)f(\eta)}+\abs{S^h(t)f(\eta)-S(t)f(\eta)}
        \\\
        &+\abs{S(t)f(\eta)-S^k(t)f(\eta)} 
        +\abs{S^k(t)f(\eta)-\mu^k(f)}
        \\\
        \leq \  
        &2 C(f) F(t) + c \abs{\Lambda} \norm{f}_\infty \exp(ct)\left(\Phi_{\varrho,S}(h-L) + \Phi_{\varrho,S}(k-L)\right)
    \end{align*}
    for some constant $c>0$,  where we use the notation 
    \begin{align*}
        \Phi_{\varrho,S}(r) := \sup_{x \in S}\sum_{y\in S\colon d(y,x)> r}\varrho(d(y,x)). 
    \end{align*}
    For $\varepsilon>0$ we can first choose $t>0$ sufficiently large to make the first term smaller than $\varepsilon/2$ and then use assumption $\mathbf{(R4)}$ to choose $h(\varepsilon)>0$ sufficiently large to make the second term smaller than $\varepsilon/2$ for all $k\geq h \geq h(\varepsilon)$. Thus, $\lim_{h \to \infty}\mu^h = \mu^*$ exists.  It remains to show that the limiting measure $\mu^*$ is stationary for the infinite-volume dynamics. 
    For this, let $f\colon\Omega\to\R$ be a $\Lambda$-local observable and note that for any $t\geq0$ and $h>0$ we have 
    \begin{align*}
        \big|\mu^* S(t)[f]-\mu^*[f]\big|
        \leq \ 
        &\big|\mu^*S(t)[f]-\mu^h S(t)[f]\big|+\big|\mu^hS(t)[f]-\mu^hS^h(t)[f]\big|
        \\\
        &+\big|\mu^hS^h(t)[f]-\mu^h[f]\big|+\big|\mu^h[f]-\mu^*[f]\big|. 
    \end{align*}
    By stationarity of $\mu^h$ with respect to $(S^h(t))_{t\geq 0}$ the third term vanishes and we only have to estimate the remaining three. Note that by weak convergence of $(\mu^h)_{h \geq 0}$ to $\mu^*$ the first and the fourth term go to zero as $h$ tends to infinity and the second term can again be bounded by invoking Theorem~\ref{theorem:restriction-to-finite-volumes} to get 
    \begin{align*}
        \big|\mu^hS(t)[f]-\mu^hS^h(t)[f]\big| \leq c \abs{\Lambda}\norm{f}_\infty \exp(ct)\Phi_{\varrho, S}(h-L).
    \end{align*}
    Since $t$ is arbitrary but fixed, we can again use $\mathbf{(R4)}$ to choose $h$ sufficiently large to make the right side arbitrarily small. This shows that $\mu^*$ is indeed a stationary measure for the unrestricted dynamic. 
\end{proof}

\subsection{Spatial decay of correlations}

Let us now turn our attention towards the bounds on the correlations at time $t$. 

\begin{proof}[Proof of Theorem~\ref{thm:spatial-decay-correlations-lieb-robinson}]
By applying the estimate in Lemma~\ref{lemma:ligget-estimate-2} and assumptions $\mathbf{(L1)}$ and $\mathbf{(R1)}$ we get 
\begin{align*}
    \norm{S(t)[fg] - [S(t)f][S(t)g]}_\infty 
        \leq \mathbf{C}_1
        \sum_{x,y\in S\colon d(x,y)\leq L}\int_0^t \left({\rm e}^{s\Gamma}\delta_\cdot f\right)(x)\left({\rm e}^{s\Gamma}\delta_\cdot g\right)(y){\rm d}s 
\end{align*}
and it thus again reduces our problem to understanding the behaviour of the $\Gamma$-operator. In the notation of the previous section, we can write 
\begin{align*}
    \left(e^{s\Gamma}\delta_\cdot h \right)(z) = \sum_{u \in \Z^d}\gamma_s(u,z)\delta_u h, 
\end{align*}
for any $h \in D(\Omega)$, $z \in \Z^d$ and $s\geq 0$. 
By non-negativity we can exchange the order of summation and integration to get 
\begin{align*}
    \norm{S(t)[fg] - [S(t)f][S(t)g]}_\infty 
        \leq \mathbf{C}_1
    \sum_{u,v \in S }(\delta_u f )(\delta_v g)\sum_{x,y\in S\colon  d(x,y) \leq L}
    \int_0^t\gamma_s(u,x)\gamma_s(v,y){\rm d}s. 
\end{align*}
By definition of $\vertiii{\cdot}$ we have $\vertiii{h}=\sum_{x \in \Z^d}\delta_x h$, so it suffices to show that for any $s\geq 0$ and $u,v\in \Z^d$ we have 
\begin{align}\label{ineq:fixed-time-estimate}
    \sum_{x,y\in S\colon d(x,y) \leq L}\gamma_s(u,x)\gamma_s(v,y)
    \leq 
    \frac{\varrho(d(u,v))}{\varrho(L)\mathbf{C}_\varrho}\exp(2\mathbf{C}_\varrho \mathbf{C}_\gamma s). 
\end{align}
Indeed, if we have \eqref{ineq:fixed-time-estimate} we obtain
\begin{align*}
    \sum_{u,v \in S}(\delta_u f )(\delta_v g)&\sum_{x,y \in S\colon d(x,y) \leq L}\int_0^t\gamma_s(u,x)\gamma_s(v,y){\rm d}s
    \\\
    \leq \ 
    &\frac{1}{\varrho(L)\mathbf{C}_\varrho^2 \mathbf{C}_\gamma}
     \sum_{u,v \in S}(\delta_u f )(\delta_v g)\int_0^t \mathbf{C}_\varrho \mathbf{C}_\gamma {\rm e}^{2\mathbf{C}_\varrho \mathbf{C}_\gamma s}\varrho(d(u,v)){\rm d}s 
     \\\
     =
     \ 
     &\frac{1}{\varrho(L)\mathbf{C}_\varrho^2 \mathbf{C}_\gamma}
     \vertiii{f}\vertiii{g}{\rm e}^{2\mathbf{C}_\varrho \mathbf{C}_\gamma t}\varrho( \text{dist}(\Lambda_f, \Lambda_g)). 
\end{align*}
To show \eqref{ineq:fixed-time-estimate}, note that for fixed $u,v\in S$ and  $L,s>0$ we get 
\begin{align*}
    \sum_{x,y\in S\colon d(x,y) \leq L}\frac{\gamma_s(u,x)\gamma_s(v,y)}{\varrho(d(u,v))}
    &\leq 
    \mathbf{C}_\varrho \sum_{x,y\in S\colon  d(x,y) \leq L}\frac{\gamma_s(u,x)\gamma_s(v,y)}{\varrho(d(u,x))\varrho(d(x,y))\varrho(d(v,y))}
    \\\
    &\leq 
    \frac{\mathbf{C}_\varrho}{\varrho(L)}\sum_{x\in S}\frac{\gamma_s(u,x)}{\varrho(d(u,x))}\sum_{y \in S}\frac{\gamma_s(v,y)}{\varrho(d(v,y))}
    \\\
    &\leq 
    \frac{{\rm e}^{2\mathbf{C}_\varrho \mathbf{C}_\gamma s}}{\varrho(L)\mathbf{C}_\varrho}.
\end{align*}
Rearranging this yields the claimed estimate~\eqref{ineq:fixed-time-estimate} and we are done. 
\end{proof}

Now we can proceed to use the bound on the speed at which information spreads as stated in Theorem~\ref{thm:spatial-decay-correlations-lieb-robinson} to derive some information about the limiting measures. 

\begin{proof}[Proof of Theorem \ref{theorem:decay-of-correlations-limiting-measure}]
    First note that since adding a constant to $f$ or $g$ has no effect on both the left and the right side of the inequality, we can assume without loss of generality that $\nu(f) = \nu(g) = 0$. This in particular implies that $f$ and $g$ cannot be identically equal to some non-zero constant and hence 
    \begin{align}\label{proof:decay-of-correlation-equilibrium-1}
        \norm{f}_\infty \leq \vertiii{f}, \quad \norm{g}_\infty \leq \vertiii{g}, \quad \text{and} \quad \vertiii{fg} \leq 2\vertiii{f}\vertiii{g}. 
    \end{align}
    For $0<s<t$ we can write 
    \begin{align*}
        \big|S(t)(fg)(\eta)& - [S(t)f(\eta)][S(t)g(\eta)]\big|
        \\
        \leq 
        &\big|S(s)(fg)(\eta) - [S(s)f(\eta)][S(s)g(\eta)]\big|
        +
        \big|S(t)(fg)(\eta) - S(s)(fg)(\eta)\big|\\
        &\quad
        +
        \big|S(t)f(\eta)\big|\big|S(t)g(\eta) - S(s)g(\eta)\big|
        +
        \big|S(s)g(\eta)\big|\big|S(t)f(\eta)-S(s)f(\eta)\big|. 
    \end{align*}
    For the last three terms we can use~\eqref{assumption:rapid-convergence}, while the first term will be estimated using Theorem~\ref{thm:spatial-decay-correlations-lieb-robinson}. Together with~\eqref{proof:decay-of-correlation-equilibrium-1} this yields
    \begin{align*}
        \big|S(s)(fg)(\eta) - [S(s)f(\eta)][S(s)g(\eta)]\big|
        \leq 
        \mathbf{C} \vertiii{f}\vertiii{g}\Big[{\rm e}^{\mathbf{C}_\varrho\mathbf{C}_\gamma s}\frac{\varrho(\text{dist}(\Lambda_f, \Lambda_g)}{\varrho(L)\mathbf{C}_\varrho} + 8 {\rm e}^{-\delta s}\Big], 
    \end{align*}
    where $\mathbf{C}:= \max\{\mathbf{C}_1,\hat{K}\}$. Now, we still have freedom in choosing $s$ appropriately to optimise the bound over $0\leq s \leq t$. A brief calculation yields that the optimal $s^*$ is 
    \begin{align*}
        s^* = \frac{1}{\mathbf{C}_\varrho \mathbf{C}_\gamma + \delta}\log\Big(\frac{8 \delta \varrho(L)}{\mathbf{C}_\gamma \varrho(\text{dist}(\Lambda_f, \Lambda_g))}\Big).
    \end{align*}
    So for $t\geq s^*$ we can plug this in to obtain the desired bound with constants given by 
    \begin{align*}
        K &:= \mathbf{C}(\mathbf{C}_\varrho \mathbf{C}_\gamma + \delta)(\mathbf{C}_\varrho^2 \mathbf{C}_\gamma \varrho(L))^{-\frac{\delta}{\mathbf{C}_\varrho \mathbf{C}_\gamma + \delta}}\left(8/\delta\right)^{\frac{\mathbf{C}_\varrho\mathbf{C}_\gamma}{\mathbf{C}_\varrho \mathbf{C}_\gamma + \delta}}\quad\text{ and }\quad
        \alpha := \delta/(\mathbf{C}_\varrho \mathbf{C}_\gamma + \delta). 
    \end{align*}
    For $t < s^*$, the bound follows directly from Theorem~\ref{thm:spatial-decay-correlations-lieb-robinson}. 
\end{proof}

\section{The strong attractor property}\label{section:proof-refined-restriction}
For reasons that will become clear later, see Lemma~\ref{lemma:minimal-cost}, we will need the following extension of Theorem~\ref{theorem:restriction-to-finite-volumes} to time-dependent restrictions. We will only make use of this result for interacting particle systems on $\Z$ but state and prove it for arbitrary dimensions $d\in \N$. 
Let $\Lambda \Subset \Z^d$, fix a non-decreasing function $h\colon [0,\infty) \to (0,\infty)$ and define time-dependent generators by 
\begin{align*}
    \mathscr{L}^{h,\Lambda}_s f(\eta) = \sum_{\Delta \subset \Lambda^{h(s)}}\sum_{\xi_\Delta\in \Omega_\Delta}c_\Delta(\eta,\xi_\Delta)[f(\xi_\Delta \eta_{\Delta^c})-f(\eta)], \quad \eta \in \Omega, f \in D(\Omega). 
\end{align*}
The associated flow on $C(\Omega)$ will be denoted by $(S^{h,\Lambda}_{s,t})_{0\leq s\leq t}$ and we will also use the notation $S^{h,\Lambda}(t):=S^{h,\Lambda}_{0,t}$.

\begin{proposition}[Refined restriction estimate]\label{proposition:refined-restriction-estimate}
Assume that the conditions $\mathbf{(L1)}$ and $\mathbf{(R1)-(R3)}$ are satisfied for $\varrho(r) = \exp(-\alpha r)$ for some $\alpha >0$. Let $\Lambda \Subset \Z^d$ and consider the time-dependent restrictions $(\mathscr{L}^{h,\Lambda}_s)_{s \geq 0}$ for $h\colon [0,\infty) \to (0,\infty)$ defined by 
\begin{align*}
    h(s) = \frac{2\mathbf{C}_\varrho \mathbf{C}_\gamma}{\alpha}(t-s) + L + k. 
\end{align*} 
Then, for any $k>(d-1)/\alpha$, there exists a constant $C=C(d,\alpha, \mathbf{C}_1,\mathbf{C}_\gamma, \mathbf{C}_\varrho, L) > 0$ such that for any initial distribution $\mu \in \calM_1(\Omega)$ the total variation error in $\Lambda$ is bounded as
\begin{align*}
    d_{\emph{TV},\Lambda}\big(\mu S(t), \mu S^h(t)\big) \leq C \abs{\Lambda} {\rm e}^{-\alpha k} k^{d-1}.
\end{align*}
\end{proposition}

\begin{proof}
    By Duhamel's formula we have 
    \begin{align*}
        S^{h,\Lambda}_{0,t}f(\eta) - S_t f(\eta) = \int_0^t \Big(S^{h}_{0,s}(\mathscr{L}^{h,\Lambda}_s - \mathscr{L})S_{t-s}\Big)f(\eta) {\rm d}s. 
    \end{align*}
    Now for any $g \in D(\Omega)$ we can use a telescoping trick to obtain the uniform estimate 
    \begin{align*}
        \big|(\mathscr{L}^{h,\Lambda}_s - \mathscr{L})g(\eta)\big| 
        \leq 
        \sum_{\Delta \not\subset \Lambda^{h(s)}}\sum_{\xi_\Delta\in \Omega_\Delta}\abs{c_\Delta(\eta, \xi_\Delta)\left[g(\xi_\Delta\eta_{\Delta^c})-g(\eta)\right]}
        \leq 
         C(d,L,\mathbf{C}_1) \sum_{x \notin \Lambda^{h(s)-L}}\delta_x g.
    \end{align*}
    After using that for any $h\in C(\Omega)$ we have $\Vert S_{0,s}^{h,\Lambda} h\Vert_\infty \leq \norm{h}_\infty$, one can apply the above estimate to the function $g=S_{t-s}f$ to see that 
    \begin{align*}
        \sup_{\eta}\big|S_{0,t}^{h,\Lambda}f(\eta) - S_t f(\eta)\big|
        &\leq 
        \int_0^t \sup_{\eta}\big|\big(S^{h,\Lambda}_{0,s}(\mathscr{L}^{h,\Lambda}_s - \mathscr{L})S_{t-s}\big)f(\eta)\big|{\rm d}s 
        \\\
        &\leq 
        \int_0^t \sup_{\eta}\big|\big((\mathscr{L}^{h,\Lambda}_s - \mathscr{L})S_{t-s}\big)f(\eta)\big|{\rm d}s
        \\\
        &\leq 
        C(d,L,\mathbf{C}_1)
        \int_0^t \sum_{x \notin \Lambda^{h(s)-L}}\delta_x (S_{t-s}f) {\rm d}s. 
    \end{align*}
    Here we can now use that, in the case where the dependence of the transition rates decays exponentially, a combination of Lemma~\ref{lemma:regularity-estimate-via-gamma} and Lemma~\ref{lemma:quantitative-spreading-bound-gamma} yields
    \begin{align*}
        \delta_x(S_{t-s}f) 
        \leq 
        \exp((t-s)\Gamma)[\delta_\cdot f](x)
        \leq 
        \norm{f}_\infty \mathbf{C}_\varrho^{-1}\exp\big(\mathbf{C}_\gamma\mathbf{C}_\varrho(t-s) - \alpha \cdot \text{dist}(x,\Lambda)\big)
    \end{align*}
    and hence 
    \begin{align*}
        \int_0^t \sum_{x \notin \Lambda^{h(s)-L}}\delta_x (S_{t-s}f) ds
        \leq 
        \norm{f}_\infty \mathbf{C}_\varrho^{-1}\int_0^t \sum_{x \notin \Lambda^{h(s)-L}} \exp\big(\mathbf{C}_\gamma\mathbf{C}_\varrho(t-s) - \alpha \cdot \text{dist}(x,\Lambda)\big) {\rm d}s. 
    \end{align*}
    We first upper bound the sum by using that there are $O(r^{d-1})$ points at distance equal to $r$ for any given site $x \in \Z^d$, i.e., 
    \begin{align*}
        \int_0^t\sum_{x \notin \Lambda^{h(s)-L}} &\exp\big(\mathbf{C}_\gamma\mathbf{C}_\varrho(t-s) - \alpha \cdot \text{dist}(x,\Lambda)\big) {\rm d}s 
        \\\
        \leq
        \
        &C(d)\abs{\Lambda}
        \int_0^t \exp\big(\mathbf{C}_\gamma \mathbf{C}_\varrho(t-s)\big)\sum_{r \geq h(s)-L}\exp(-\alpha r)r^{d-1} {\rm d}s. 
    \end{align*}
    Now we can use that the function $r \mapsto \exp(-\alpha r) r^{d-1}$ is non-increasing  on the interval $((d-1)/\alpha, \infty)$, so by choosing $k$ sufficiently large we can estimate the sum by an integral over a slightly larger domain to get 
    \begin{align*}
        \int_0^t\sum_{x \notin \Lambda^{h(s)-L}} &\exp\big(\mathbf{C}_\gamma\mathbf{C}_\varrho(t-s) - \alpha \cdot \text{dist}(x,\Lambda)\big) {\rm d}s
        \\\
        &\leq 
        C(d) \int_0^t \exp\big(\mathbf{C}_\gamma \mathbf{C}_\varrho(t-s)\big) \int_{h(s)-L-1}\exp(-\alpha r)r^{d-1}{\rm d}r {\rm d}s
        \\\
        &\leq 
        C(d, \alpha) \int_0^t \exp\big(\mathbf{C}_\gamma \mathbf{C}_\varrho(t-s)\big)\int_{\alpha(h(s)-L-1)}^\infty u^{d-1}\exp(-u) {\rm d}u {\rm d}s, 
    \end{align*}
    where we applied a change of variable $u = \alpha^{-1} r$ in the inner integral to get the last inequality. Note that the undefined constants may vary from line to line. By using the recursion formula for the upper incomplete Gamma function, see Lemma~\ref{lemma:elementary-gamma-function-estimate} below for details, this can in turn be estimated as 
    \begin{align*}
        \int_0^t &\exp\big(\mathbf{C}_\gamma \mathbf{C}_\varrho(t-s)\big)\int_{\alpha(h(s)-L-1)}^\infty u^{d-1}\exp(-u) {\rm d}u {\rm d}s
        \\\
        &\leq
        C(d,\alpha) \int_0^t \exp\big(\mathbf{C}_\gamma \mathbf{C}_\varrho(t-s)-\alpha(h(s)-L-1)\big)\big[\alpha(h(s)-L-1)\big]^{d-1} {\rm d}s.
    \end{align*}
    Now, since we chose
    \begin{align*}
        h(s) = 2\frac{\mathbf{C}_\gamma\mathbf{C}_\varrho}{\alpha}(t-s)+ L +1 + k, 
    \end{align*}
    where $k>0$ is some constant that is assumed to be sufficiently large to make use of the previously mentioned monotonicity, we get
    \begin{align*}
       \int_0^t &\exp\big(\mathbf{C}_\gamma \mathbf{C}_\varrho(t-s)-\alpha(h(s)-L-1)\big)\big[\alpha(h(s)-L-1)\big]^{d-1} {\rm d}s
       \\\
       &\leq 
       C(\alpha,d, \mathbf{C}_\gamma, \mathbf{C}_\varrho) {\rm e}^{-\alpha k} k^{d-1} \int_0^t \exp\big(-\mathbf{C}_\gamma \mathbf{C}_\varrho(t-s)\big)(t-s)^{d-1} {\rm d}s
        \\\
        &\leq 
         C(\alpha,d, \mathbf{C}_\gamma, \mathbf{C}_\varrho)  {\rm e}^{-\alpha k} k^{d-1}. 
    \end{align*}
    Putting everything together yields the claimed upper bound. 
\end{proof}

In the above proof we used the following elementary estimate for the Gamma function. 

\begin{lemma}\label{lemma:elementary-gamma-function-estimate}
    For $d \in \N$ and $x>0$, the upper incomplete Gamma functions defined by 
    \begin{align*}
        \Gamma(d,x) = \int_x^\infty r^{d-1}{\rm e}^{-r}{\rm d}r,
    \end{align*}
    satisfies the upper bound
    \begin{align*}
        \Gamma(d,x) \leq {\rm e} (d-1)! \, {\rm e}^{-x} x^{d-1}. 
    \end{align*}
\end{lemma}
    
\begin{proof}
    Via integration-by-parts one obtains the recurrence relation
    \begin{align*}
        \Gamma(n+1,x) = n\cdot \Gamma(n,x) + x^n {\rm e}^{-x}
    \end{align*}
    and using this inductively yields the explicit formula 
    \begin{align*}
        \Gamma(d,x) = (d-1)!{\rm e}^{-x}\sum_{n=0}^{d-1}\frac{x^n}{n!}, 
    \end{align*}
    which directly yields the claimed estimate. 
\end{proof}

\section{Relative entropy, speed-up, and time-shift}\label{section:proof-speed-up}
For two probability laws $\mathbf{P}, \mathbf{Q}$ on a measurable space $(\mathbb{X}, \mathcal{X})$ we define the \textit{relative entropy of $\mathbf{P}$ with respect to $\mathbf{Q}$} by 
\begin{align*}
    H(\mathbf{P}\lvert \mathbf{Q}) 
    =
    \begin{cases}
        \int_\mathbb{X} \log({\rm d}\mathbf{P}/{\rm d}\mathbf{Q}) {\rm d}\mathbf{P}, \quad &\text{if } \mathbf{P} \ll \mathbf{Q}, 
        \\\
        \infty, &\text{otherwise. }
    \end{cases}
\end{align*}

We begin by stating the following Girsanov-type formula, which we will use to compare the sped-up process with the original dynamics. 

\begin{lemma}[Girsanov formula]\label{lemma:girsanov-transformation}
   Consider a continuous-time Markov chain with time-inhomogeneous generator $(L_s)_{s \geq 0}$ on a finite state space $\mathcal{X}$ and let $(\hat{L}_s)_{s \geq 0}$ be the generator of another continuous-time Markov chain such that for every $s\geq 0$ the transition rates of $L_s$ and $\hat{L}_s$ satisfy the condition
   \begin{align*}
       L_s(x,y) = 0 \Rightarrow \hat{L}_s(x,y) = 0,\quad \forall x,y \in \mathcal{X}.
   \end{align*}
   Additionally, assume that the set $D$ of discontinuity points of the set of transition rates $\{L_\cdot(x,y)\colon x,y\in \mathcal{X}\} \cup \{\hat{L}_\cdot(x,y)\colon x,y\in \mathcal{X}\}$ has zero Lebesgue measure. 
   Denote the induced path measures on the space of $\mathbb{X}$-valued c\'adl\'ag paths $X([0,t])$ up to time $t >0$ by $\mathbb{Q}_x$ respectively $\hat{\mathbb{Q}}_x$, where the initial condition $X(0) = x$ is deterministic. Then, the following Girsanov-type formula holds 
   \begin{align*}
    \frac{{\rm d} \mathbb{Q}_x}{{\rm d}\hat{\mathbb{Q}}_x} \big( X([0,t])\big)
    =
    \exp\Big(-\int_0^t \delta_s(X(s)) + \sum_{s \in [0,t]\colon X(s_-) \neq X(s)}\sum_{y \neq X(s)}\log\tfrac{L_s(X(s_-), X(s))}{\hat{L}_s(X(s_-), X(s))}\Big){\rm d}s, 
\end{align*}
where 
\begin{align*}
    \delta_s(x) :=\sum_{y\neq x}\big(L_s(x,y) - \hat{L}_s(x,y)\big).
\end{align*}
\end{lemma}

See for example \cite[Proposition~2.6. in Appendix~1]{kipnis_scaling_1999} for a proof of the Girsanov formula for continuous-time Markov chains in the time-homogeneous case. The extension to inhomogeneous transition rates follows along similar lines but is somewhat tedious, so we omit it here. \medskip 

We can now use this to obtain bounds for the relative entropy on path space.
\begin{lemma}[Entropic cost of speed-up]\label{lemma:entropic-cost}
    Let $\mathcal{X}$ be a finite set and $(L_s)_{s \geq 0}$ generators of a time-inhomogeneous continuous-time Markov chain $(X(t))_{t\geq 0}$. Define the maximal rate at time $s$ via
    \begin{align*}
        \mathbf{c}(s) := \max_{x \in \mathcal{X}}\sum_{y\neq x}L_s(x,y).
    \end{align*}
    Let $t, \tau > 0$ and $\lambda\colon [0,t] \to (0,\infty)$ a bounded and measurable function. Denote by $(L^\lambda_s)_{s\geq 0}$ the generators of the sped-up process, i.e., 
    \begin{align*}
        L^\lambda_s f(x) = \sum_{y \in \mathcal{X}}L_s(x,y)[f(y)-f(x)], \quad x \in \mathcal{X},\ f\colon\mathcal{X}\to \R. 
    \end{align*}
    We assume that the set $D \subset  [0,\infty)$ of discontinuity points of the set of jumps rates $\{L_\cdot(x,y)\colon  x,y\in \mathcal{X}\}$ and $\lambda(\cdot)$ has zero Lebesgue measure. 
    For some fixed initial distribution $\rho \in \mathcal{M}_1(\mathcal{X})$ let $P$ respectively $P^\lambda$ denote the law of the associated Markov process with generator $L$ respectively $L^\lambda$. Then, we have 
    \begin{align}
        H(P^{\lambda} \lvert P) \leq \int_0^t \mathbf{c}(s)\big(\lambda(s)-1\big)^2 {\rm d}s. 
    \end{align}
\end{lemma}

\begin{proof}
    We start by calculating the Radon--Nikodym density via the Girsanov formula
    \begin{align*}
        \frac{{\rm d}P^\lambda}{{\rm d}P}(X([0,t])) 
        =
        \exp\Big(\sum_{s \in [0,t]\colon X_{s_-} \neq X_s}\log\lambda(s) - \int_0^t\big(\lambda(s)-1\big)\sum_{y\neq X_s}L(X_s,y){\rm d}s\Big).
    \end{align*}
    By definition of the relative entropy and $\mathbf{c}(\cdot)$ this implies 
    \begin{align*}
        H(P^\lambda \lvert P) 
        &=
        \int\log\Big(\frac{{\rm d}P^\lambda}{{\rm d}P}(X([0,t]))\Big)P^\lambda({\rm d}X[0,t]) 
        \\\
        &=\int_0^t\Big(\int  \sum_{y \neq X_s}L(X_s, y) P^\lambda({\rm d}X([0,t]))\Big)\big[\lambda(s)\log\big(\lambda(s)\big)-\big(\lambda(s)-1\big)\big] {\rm d}s
        \\\
        &\leq 
        \int_0^t \mathbf{c}(s) \big[\lambda(s)\log\big(\lambda(s)\big)-\big(\lambda(s)-1\big)\big] {\rm d}s
        \\\
        &\leq \int_0^t \mathbf{c}(s) \big(\lambda(s)-1\big)^2 {\rm d}s,
    \end{align*}
    where we used that $\log x \leq x -1$ for all $x >0$ in the last step. 
\end{proof}

Now if $\mathbf{c}(\cdot)$ grows linearly, then by choosing a constant speed-up function $\lambda \equiv \left(1+\tau/t\right)$ we would only get $H(P^\lambda\lvert P) \lesssim \tau^2$, which does not allow us to conclude anything about the $t\uparrow \infty$ limit.
But we can still optimise over the speed-up to bring us back into the game. 
\begin{lemma}[Minimal cost]\label{lemma:minimal-cost}
    Denote the set of admissible speed-ups by
    \begin{align*}
        \mathscr{H}_{t,\tau} := \Big\{\lambda \in L^1([0,t])\colon \int_0^t \big(\lambda(s)-1\big){\rm d}s = \tau \Big\}. 
    \end{align*}
    Then, for any $f\colon [0,t]\to (0,\infty)$ such that $f, f^{-1} \in L^1([0,t])$ we have 
    \begin{align*}
        \inf_{\lambda \in \mathscr{H}_{t,\tau}}\int_0^t f(s)\big(\lambda(s)-1\big)^2 {\rm d}s 
        =
        \tau^2 \Big(\int_0^t \frac{1}{f(s)}{\rm d}s\Big)^{-1}
    \end{align*}
    and the infimum is attained at 
    \begin{align*}
        \lambda^*(s) = 1 + \tau \Big(\int_0^t\frac{f(s)}{f(r)}{\rm d}r\Big)^{-1}, \quad s \in [0,t]. 
    \end{align*}
\end{lemma}

\begin{proof}
    By using the formalism of convex optimisation with constraints, the problem can be boiled down to determining the critical points of the Lagrangian 
    \begin{align*}
        \mathcal{L}(f,\gamma) = \int_0^t f(s)\big(\lambda(s)-1\big)^2 {\rm d}s - \gamma\Big( \int_0^t\big(f(s)-1\big){\rm d}s - \tau\Big), \quad \gamma \in \R,\ \lambda \in \mathscr{H}_{t,\tau}. 
    \end{align*}
    It therefore suffices to determine $\lambda(\cdot)$ such that 
    \begin{align*}
        2\int_0^t f(s)(\lambda(s)-1){\rm d}s - \gamma t = 0. 
    \end{align*}
    An elementary calculation shows that this can be done by choosing 
    \begin{align*}
        \lambda(s) = 1 + \gamma/(2 f(s)). 
    \end{align*}
    It remains to determine the correct value for the Lagrange multiplier $\gamma$ to make sure the constraint $\int_0^t(\lambda(s)-1){\rm d}s = \tau$ is satisfied. 
    This yields the equation
    \begin{align*}
        \gamma = 2 \tau \Big(\int_0^t \frac{1}{f(s)}{\rm d}s\Big)^{-1}
    \end{align*}
    and plugging this in gives precisely the claimed formula for the minimiser and the minimum. 
\end{proof}

\begin{remark}
The optimality statement in Lemma~\ref{lemma:minimal-cost} directly tells us that this approach can only work if the maximal rate $\mathbf{c}(s)$ of leaving a particular state at time $s$ does not grow too fast, since we need that 
\begin{align*}
    \lim_{t \to \infty}\int_0^t \frac{1}{\mathbf{c}(s)}{\rm d}s = \infty. 
\end{align*}
This rules out an application of this method for dimensions $d > 1$ where one should expect $\mathbf{c}(s) \sim s^d$. 
One could however extend the result to graphs whose volume grows just a tiny bit faster than $\Z$ so that the integral still blows up as $t\to\infty$.
\end{remark}

Let us finally put everything together and provide the proof of the strong attractor property in dimension one. 
\begin{proof}[Proof of Theorem \ref{theorem:main-result}]
    For every $\Lambda \Subset \Z$, the triangle inequality implies  
\begin{align*}
    d_{\text{TV}, \Lambda}(\mu_t,\mu_{t+\tau})
    \leq 
    d_{\text{TV}, \Lambda}(\mu_t,\mu_t^h)
    +
    d_{\text{TV}, \Lambda}(\mu_t^h,\mu^h_{t+\tau})
    +
    d_{\text{TV}, \Lambda}(\mu^h_{t+\tau},\mu_{t+\tau}).
\end{align*}
The first and the third term can be estimated by using Proposition~\ref{proposition:refined-restriction-estimate} and Pinsker's inequality allows us to bound the second term via 
\begin{align*}
    d_{\text{TV}, \Lambda}(\mu_t^h, \mu_{t+\tau}^h) \leq \sqrt{\frac{1}{2}H(\mu_t^{h,\lambda}\lvert \mu_t^h)}. 
\end{align*}
So by plugging in the corresponding estimates we get 
\begin{align*}
    d_{\text{TV}, \Lambda}(\mu_t,\mu_{t+\tau})
    \leq 
    2 C \abs{\Lambda}{\rm e}^{-\alpha k} k^{d-1} + \frac{\tau}{\sqrt{2}}\Big(\int_0^t \frac{1}{2 h(s)}{\rm d}s\Big)^{-1/2}, 
\end{align*}
where $h(s) = cs + L + k$ for some fixed $c>0$. 
By first sending $t$ to infinity we see that for any  sufficiently large $k>0$
\begin{align*}
    \limsup_{t \to \infty} d_{\text{TV}, \Lambda}(\mu_t,\mu_{t+\tau}) \leq 2 C \abs{\Lambda}{\rm e}^{-\alpha k} k^{d-1} =: F(k),
\end{align*}
but since $k$ is arbitrary and $F(k) \to 0$ as $k\uparrow \infty$, the limit must be equal to $0$. 
\end{proof}

\section{Outlook}\label{section:outlook}
From the results in \cite{jahnel_long-time_2025, koppl_absence_2026} and the long-range construction in \cite{jahnel_time-periodic_2025}, we expect that for $S=\Z$, even for interacting particle systems with $\gamma(x,y) \leq \abs{x-y}^{-\alpha}$ and $\alpha >2$, time-periodic behaviour should be impossible, yet the method used in this article fails in this regime. 
This limitation is mainly due to the fact that the relative-entropy bound in Lemma~\ref{lemma:entropic-cost} and Lemma~\ref{lemma:minimal-cost} requires a linear growth of the region in which the particles participate in the dynamics. However, if the interaction strength only decays like a power law, the error bound in Theorem~\ref{theorem:restriction-to-finite-volumes} does not decay if $h(t) \sim c t$ for some constant $c>0$. 
\medskip 

Unfortunately, there is little hope that refining the method based on an analysis of the operator $\Gamma$ can yield results in the power-law regime. Indeed, let us assume for a moment that the rates are translation invariant. Then, we have $\gamma(x,y) = \hat{\gamma}(x-y)$ for some function $\hat\gamma\colon\Z^d \to [0,\infty)$ and denoting the operator norm of $\Gamma$ by $M$ we can write ${\rm e}^{t\Gamma} = {\rm e}^{tM}{\rm e}^{tQ}$, where $Q$ is the generator of a random walk on $\Z^d$ with transition rates given by $(\hat{\gamma}(x))_{x \in \Z^d}$. 
The first factor always gives us exponential growth in $t$, but heat-kernel bounds for random walks with heavy-tailed jump kernel tell us that the second part cannot compensate this exponential growth by decaying sufficiently fast in space. 
This back of the envelope calculation suggests that one cannot use the strategy in this manuscript to extend Theorem~\ref{theorem:main-result} to regimes with power-law decay. 
We do however believe, that the result is still true for such systems, at least when $\alpha > 2$. For $\alpha \in (1,2)$ there are counterexamples, see~\cite{jahnel_time-periodic_2025}. 
\medskip 

A setting that can be used as a test case and where one has some more tools available is if one considers interacting particle systems with  "random-range" interactions. As an example, consider the transition rates as in \eqref{example-ising-glauber} but for every update, first sample the range of the Hamiltonian from some radius distribution $\mu \in \calM_1([0,\infty))$ and then use the truncated Hamiltonian to resample the spin. One can then first extend the graphical representation in \cite[Chapter~4]{swart_course_2026} to this setting and use a discrete version of the model in \cite{deijfen_asymptotic_2003} and ideas  from \cite{gouere_continuous_2008} and \cite{cox_greedy_1993} to prove that information spreads at linear speed as long as the radius distribution satisfies $\mu({r}) \sim r^{-\alpha}$ for $\alpha > 2d+1$. 
Note that this is the threshold at which the speed in long-range first passage percolation changes from linear to polynomial, see \cite{chatterjee_multiple_2016}, it is therefore not so surprising that $\alpha > 2d+1$ is the best one can do. The strategy described above will be carried out in \cite{jahnel-koeppl-upcoming}.

\section*{Acknowledgments} BJ and JK received support by the Leibniz Association within the Leibniz Junior Research Group on \textit{Probabilistic Methods for Dynamic Communication Networks} as part of the Leibniz Competition (grant no.\ J105/2020). 
BJ is also funded by the Deutsche Forschungsgemeinschaft (DFG, German Research Foundation) under Germany's Excellence Strategy -- The Berlin Mathematics Research Center MATH+ (EXC-2046/1, EXC-2046/2, project ID: 390685689) through the projects \emph{EF45-3} on \emph{Data Transmission in Dynamical Random Networks} and  \emph{EF-MA-Sys-2} on \emph{Information Flow \& Emergent Behavior in Complex Networks}, as well as the SPP2265 Project P27 {\em Gibbs point processes in random environment}. 

\bibliography{refs}
\bibliographystyle{alpha}

\end{document}